\documentclass[11pt]{amsart}

\usepackage{amsmath, amssymb, bm, color}  
\usepackage{amsthm} 

\newtheorem*{Theorem}{Theorem}

\newtheorem*{Corollary}{Corollary}

\theoremstyle{definition}

\newtheorem*{Definition}{Definition}

\theoremstyle{remark}

\newtheorem{Remark}{Remark}
\newtheorem*{Example}{Example}

\newcommand{\real}{\mathbb{R}}

\newcommand{\integer}{\mathbb{Z}}

\newcommand{\OP}{\mathrm{Op}}
\newcommand{\cW}{\mathcal{W}}

\begin{document}

\title{Virtually expanding dynamics}
\author[M.~Tsujii]{Masato TSUJII}
\address{Department of Mathematics, Kyushu University, Fukuoka, 819-0395}
\email{tsujii@math.kyushu-u.ac.jp}
\thanks{This work was supported by JSPS Grant-in-Aid for Scientific Research (B) 15H03627. The author thanks Hesam Rajabzadeh (IPM, Iran),  Roberto Castorrini (GSSI, Italy) and the anonymous referee for pointing out a few errors in the previous versions of this paper. }
\date{\today}

\begin{abstract}
We introduce a class of discrete dynamical systems that we call \emph{virtually expanding}. This is an open subset of self-covering maps on a closed manifold which contains all expanding maps and some partially hyperbolic volume-expanding maps.  We show that the Perron-Frobenius operator is quasi-compact on a Sobolev space of positive order for such class of dynamical systems.
\end{abstract}

\maketitle

\section{Introduction}
Expanding maps are a class of discrete dynamical systems that exhibit typically chaotic behavior of orbits and studied extensively as a standard model of chaotic dynamical systems. 
Among others, an important fact about expanding dynamical system is that the associated Perron-Frobenius operators are quasi-compact on the space of functions with some smoothness. (See \cite{Ruelle76, MR1793194}.) 
Based on the recent understanding on expanding and hyperbolic dynamical systems in terms of Fourier analysis \cite{BaladiTsujii07, GouezelLiverani06, FRS, 1706.09307}, this fact can be understood as a consequence of  the property of the Perron-Frobenius operators that they reduce high-frequency components of functions by transferring them to low-frequency ones. But actually in order to apply this argument and obtain the conclusion of quasi-compactness, we only need to have that property ``in average with respect to preimages" and ``direction-wise".  Indeed we find that the argument about quasi-compactness of the Perron-Frobenius operators for expanding maps can be extended to a much more general class of dynamical systems, which we would like to call \emph{virtually expanding maps}. 
Below we present the definition and the main property of virtually expanding maps. In the last section, we discuss briefly about a further generalization of the class of virtually expanding maps.

\section{Main result}
Let $M$ be a closed connected $C^\infty$ manifold and let $f:M\to M$ be a $C^{\infty}$ covering map. We suppose that $M$ is equipped with some Riemann metric and that the Jacobian determinant $Jf(p)$ does not vanish for any $p\in M$. 
Let $T^{*}M$ be the cotangent bundle of $M$. The push-forward action of $f$ on $T^{*}M$ is written as
\[
f^{\dag}:T^{*}M\to T^{*}M, \quad
(q,\xi) \mapsto f^{\dag}(q,\xi)=(f(q),((D_{q}f)^{*})^{-1}(\xi)).
\]
For $\mu\ge 0$ and each non-zero covector $(q,\xi)\in T^{*}M$, we set
\begin{equation}\label{eq:b}
b^{\mu}((q,\xi);f)=\sum_{f^{\dag}(p,\eta)=(q,\xi)} \frac{ (|\xi|/|\eta|)^{\mu}}{|Jf(p)|}
\end{equation}
where the sum on the right-hand side denotes that over the points $(p,\eta)\in T^*M$ satisfying $f^{\dag}(p,\eta)=(q,\xi)$ and we write $|\cdot|$ for the induced norm on the cotangent bundle from the Riemann metric.
Then we define
\begin{equation}\label{eq:B}
B^{\mu}(f)=\sup\{\, b^{\mu}((q,\xi);f)\mid 0\neq (q,\xi) \in T^{*}M\}.
\end{equation}
Note that, for another $C^{\infty}$ covering map $g:M\to M$ with non-vanishing Jacobian determinant, we have
\[
B^{\mu}(f\circ g)\le B^{\mu}(f)\cdot B^{\mu}(g).
\]
In particular $B^\mu(f^n)$ is sub-multiplicative with respect to $n$:
\begin{equation}\label{eq:submulti}
B^{\mu}(f^{n+m})\le B^{\mu}(f^{n})\cdot B^{\mu}(f^{m})\quad \quad \mbox{for }n,m\ge 1.
\end{equation}
\newcommand{\qe}{\mathcal{VE}}
\begin{Definition}
For $\mu>0$, a $C^{\infty}$ covering map $f:M\to M$ is $\mu$-virtually expanding if 
\[
\inf_{n\ge 1}(B^{2\mu}(f^{n}))^{1/n}=\lim_{n\to \infty} (B^{2\mu}(f^{n}))^{1/n}<1,
\]
where the equality holds by Fekete subadditive lemma. 
We write $\qe^{\mu}(M)$ for the set of $\mu$-virtually expanding maps on $M$. 
We call an element of  $\qe(M)=\cup_{\mu>0} \qe^\mu(M)$ virtually expanding. 
$\qe^{\mu}(M)$ for $\mu>0$ as well as $\qe(M)$ are $C^{1}$ open subsets of $C^{\infty}(M,M)$. 
\end{Definition}

\begin{Remark}
The definition of virtually expanding property above does not depend on the choice of the Riemann metric on $M$, and we may even consider  any continuous Finsler metric in the place of the Riemann metric. 
\end{Remark}

An expanding map $f:M\to M$ is virtually expanding. Indeed by a standard distortion estimate we have that 
\[
\lim_{n\to \infty} (B^{0}(f^n))^{1/n}\le 1
\]
and hence, for any $\mu>0$, 
\[
\lim_{n\to \infty} (B^{2\mu}(f^n))^{1/n}\le \lambda^{-2\mu}<1
\]
where $\lambda>1$ is the minimum expansion rate of $f$. 

But the converse is not true except for one dimensional cases because,  in the definition \eqref{eq:b}, we consider \emph{each} (direction of) non-zero covector $(q,\xi)\in T^*M$  and take the \emph{averaged value} of the expansion rates over the backward images of $(q,\xi)$ with weight $1/|Jf|$.
There are in fact many virtually expanding maps that are not expanding.
Let us say that a  $C^\infty$ map $f:M\to M$ is volume-expanding if
\[
\lim_{n\to \infty} 
\left(\inf_{p\in M} |Jf^n(p)|\right)^{1/n}>1.
\]
A virtually expanding map is necessarily volume-expanding.
The example below shows that  the set  $\qe(\mathbb{T}^2)$ of virtually expanding maps on the $2$-dimensional torus $\mathbb{T}^2=\real^2/\integer^2$ contains a non-empty $C^1$ open subset of volume-expanding maps, which are not expanding. 
\begin{Example} 
Consider a map $f:\mathbb{T}^2\to \mathbb{T}^2$ on the torus $\mathbb{T}^2$ defined by 
\begin{equation}\label{example}
f(x,y)=(mx,y+m\cos 2\pi x)\quad \mod\integer^2
\end{equation}
with $m$ an integer. 
For this map and its small $C^1$ perturbations, we can check the condition $B^{2\mu}(f)<1$ at least if we assume $m$ to be sufficiently large depending on $\mu>0$. A proof will be provided in the appendix. 
\end{Example}
\begin{Remark}
The argument in the proof of the claim $B^\mu(f)<1$ is closely related to that for partially hyperbolic endomorphisms in \cite{MR2191383,  MR3521633, ZhangZ} using some transversality condition on the images of  unstable cones. 
It will be possible to prove that a generic $C^\infty$ partially hyperbolic volume-expanding maps are virtually expanding. Further it should not be too optimistic to expect that the same is true without the assumption of partial hyperbolicity. 
\end{Remark}
The Perron-Frobenius operator $P$ expresses the push-forward action of $f$ on densities on $M$ and is defined explicitly  by 
\[
P u(q)=\sum_{p\in f^{-1}(q)}\frac{u(p)}{|Jf(p)|}.
\]
It is not difficult to see that $P$ gives a bounded operator on the Sobolev space $H^{\mu}(M)$ of any order $\mu\in \real$. (See \cite[\S 1.3, \S 1.5]{TaylorText81} for the definition of $H^{\mu}(M)$.)
Our main theorem  is the following. 
\begin{Theorem}\label{th:main}
If $f$ belongs to $\qe^{\mu}(M)$ for some $\mu>0$, the Perron-Frobenius operator $P$ is quasi-compact on the Sobolev space $H^{\mu}(M)$ of order $\mu>0$, that is, the essential spectral radius of $P:H^{\mu}(M)\to H^{\mu}(M)$ is smaller than its spectral radius $1$. 
\end{Theorem}
\begin{Remark}
Once we show that the essential spectral radius of $P:H^\mu(M)\to H^\mu(M)$ is smaller than $1$, it is easy to see that the spectral radius equals $1$ as the spectral radius is just the maximum of the moduli of  the  peripheral discrete eigenvalues.
\end{Remark}
\begin{Remark}
In the conclusion of Theorem, 
there may be several eigenvalues of modulus $1$ and the multiplicity of the eigenvalue $1$ may be greater than $1$. It will be possible to observe various (complicated but relatively tame) bifurcations of those eigenvalues when we perturb the maps in the class $\qe^{\mu}(M)$. 
\end{Remark}
From the conclusion of the theorem, a few important dynamical properties of the map $f$ follows immediately. 
For instance we have
\begin{Corollary} 
If $f$ belongs to $\qe^{\mu}(M)$ for some $\mu>0$, there are finitely many absolutely continuous ergodic measures whose densities with respect to the smooth measure are in $H^\mu(M)$ and almost every point on $M$ with respect to the smooth measure is generic for either of them. 
\end{Corollary}
\begin{proof} We only prove the latter claim. (The proof of the former claim should be rather standard if we note that $P$ preserves positive and negative part of functions.)
We write $m$ for the Riemann volume on $M$.  
Let $\mu_i=\varphi_i dm$, $1\le i\le \ell$,  be the absolutely continuous ergodic measures whose densities $\varphi_i$ belong to $H^\mu(M)$. Suppose that the measurable subset 
\[
X=\{x\in M\mid \text{$x$ is not generic for either of $\mu_i$ for $1\le i\le \ell$}\}
\]
has positive measure with respect to $m$.
We approximate the characteristic function of $X$ by a $C^\infty$ positive-valued function $\chi$ in $L^2$ sense. From the quasi-compactness of $P$, the sequence of functions
\[
\chi_m:=\frac{1}{m}\sum_{k=0}^{m-1} P^k \chi
\] 
converges, as $m\to \infty$, to a function in the eigenspace $E$ for the eigenvalue $1$ in $H^\mu(M)$ and hence in $L^2(M)$. 
However, since $f(X)= X$ and since $P$ preserves the $L^1$ norms of positive-valued functions, the  $L^2$-norm of the component of $\chi_m$ that is $L^2$-orthogonal to the subspace $E$ is bounded from below by a positive constant, provided that $\chi$ is sufficiently close to the characteristic function of $X$ in $L^2$ sense. This is  a contradiction and therefore $X$ is a null subset.  
\end{proof}

\section{Proof of Theorem}
We will use a few basic theorems on pseudo-differential operators and notations given in \cite[Ch.1 and 2]{TaylorText81}. (In particular, see the remark at the end of \cite[\S 2.5]{TaylorText81} for the concept of a pseudo-differential operator on a manifold $M$.) 
First we take a small open connected subset $U\subset M$ so that its backward image $f^{-1}(U)$ consists of connected open subsets $V_1$, $V_2$, \dots, $V_{\deg f}$, on each of which the map $f$ is a diffeomorphism onto $U$. Here $\deg f$ denotes the mapping degree of $f$. We write $f_j:V_j\to U$ for the restriction of $f$ to $V_j$. For convenience, we suppose that each of $U$ and $V_j$, $1\le j\le \deg f$, is contained in a single coordinate chart. 

Suppose that a function $u\in C^\infty(M)$ is supported on $f^{-1}(U)=\cup_{j=1}^{\deg f}V_j$ and we write $u_j$ for the restriction of $u$ to $V_j$. Then we may write 
\[
Pu =\sum_{j=1}^{\deg f} P_j u_j, \quad P_j u_j= \frac{u_j}{Jf}\circ f_j^{-1}
\]
where $P_j$ denotes the restriction of $P$ to the functions supported on $V_j$. Note that $P_j$ viewed in the local coordinate chart is  a multiplication by a smooth function composed with an operation of change of variables. 

For $\mu\ge 0$, we define a function 
\[
\mathcal{W}^\mu:T^*M \to \real,\quad 
\mathcal{W}^\mu(x,\xi)=\langle |\xi|\rangle^{\mu} 
\]
with setting $\langle s\rangle :=\sqrt{1+s^2}$. 
This function belong to a standard symbol class $S^{\mu}_{1,0}$ (see \cite[\S  2.1]{TaylorText81}) and we may define the norm $\|\cdot\|_{H^\mu}$ on the Sobolev space $H^\mu(M)$ by 
\[
\|u\|_{H^\mu}^2=\|\OP(\mathcal{W}^\mu) u\|_{L^2}^2+ \|u\|_{L^2}^2
\]
where (and henceforth) we write\footnote{In the reference \cite{TaylorText81}, pseudo-differential operators are written as $F(x,D)$. Here  we use a different notation for convenience.}  $\OP(F)$ for the pseudo-differential operator with symbol $F$. (See \cite[\S 2.1, \S 2.5]{TaylorText81}.) For convenience we will understand $H^\mu(M)=L^2(M)$ when $\mu<0$.

We write $
Df^*_j:T^*U\to T^*V_j$
for the pull-back action of $f_j$ on the cotangent bundle restricted $T^*U$.  This is a bijection. 
Let us consider the pseudo-differential operators
\[
A_j=\OP\left(\frac{\cW^\mu}{\sqrt{Jf\circ f_j^{-1}}\cdot \cW^\mu\circ Df_j^{*}}\right)\quad \text{and}\quad 
B_j=\OP\left(\frac{\cW^\mu \circ Df_j^*}{\sqrt{Jf\circ f_j^{-1}}}\right) 
\]
and set 
\[ 
v_j:=u_j\circ f_j^{-1}\quad \text{and}\quad b_j=B_j v_j
\]
for $j=1,2,\dots, \deg f$. Then, by the product formula (\cite[\S  2.4]{TaylorText81}) of pseudo-differential operator, we have 
\begin{align*}
&\|\OP(\cW^\mu)Pu\|_{L^2}\le 
\left\|\sum_{j=1}^{\deg f} 
A_j B_j v_j \right\|_{L^2} + \sum_{j=1}^{\deg f}\|\OP(\cW^\mu) P_j u_j-A_jB_j v_j\|_{L^2}\\
&= 
\left\|\sum_{j=1}^{\deg f} 
A_j B_j v_j \right\|_{L^2} + 
\sum_{j=1}^{\deg f}
\bigg\|\bigg(\OP(\cW^\mu) \OP\bigg(\frac{1}{Jf\circ f_j^{-1}}\bigg) -A_jB_j\bigg) v_j\bigg\|_{L^2}\\
&\le 
\left\|\sum_{j=1}^{\deg f} 
A_j B_j v_j\right\|_{L^2} + C\sum_{j=1}^{\deg f}\|v_j\|_{H^{\mu-1}}
\end{align*}
where $C$ is a constant that may depend on $f$ but not on $u$.
(In the following we use $C$ as a generic symbol for constants of this kind.)  
Recall the Lagrange identity
\[
\sum_{i,j=1}^d |x_i|^2|y_j|^2- \left|\sum_{i=1}^d \bar{x}_i y_i\right|^2=\frac{1}{2} 
\sum_{i\neq j} |\bar{x}_i y_j-\bar{x}_j y_i|^2.
\]
Following its proof with regarding $A_i$ and $b_i=B_iv_i$ as $x_i$ and $y_i$ respectively\footnote{Further regard $A_i^*$ as $\bar{x}_i$ and the $L^2$ inner product as the product of complex numbers in a few places.} and  using the product and adjoint formula (\cite[\S  2.4]{TaylorText81}) to shuffle the pseudo-differential operators, we see that
\begin{align*}
\sum_{i,j=1}^{\deg f} \|A_j b_i\|_{L^2}^2 
&-
\left\|\sum_{j=1}^{\deg f} 
A_j b_j\right\|_{L^2}^2\\
&\ge \frac{1}{2}\sum_{i\neq j} \|A_j b_i-A_i b_j\|_{L^2}^2 -C \sum_{j=1}^{\deg f}\|v_j\|_{H^{\mu-(1/2)}}^2.
\end{align*}
Further, by the product and adjoint formula, we see
\[
\sum_{i,j=1}^{\deg f} \|A_j b_i\|_{L^2}^2
\le 
\sum_{i=1}^{\deg f}
\mathrm{Re} \langle b_i, A b_i\rangle
+C\sum_{j=1}^{\deg f}\|v_j\|_{H^{\mu-(1/2)}}^2
\]
where $A$ is the pseudo-differential operator with real symbol
\[
A=\OP\left(\sum_{j=1}^{\deg g} \frac{\cW^{2\mu}}{(Jf\circ f_j^{-1})\cdot \cW^{2\mu}\circ Df_j^*}\right)
\]
and we have used the estimate
\[
|\mathrm{Im} \langle b_i, A b_i\rangle| =
\left|\left\langle b_i, \frac{A-A^*}{2} b_i\right\rangle \right| \le C \|v_i\|^2_{H^{\mu-(1/2)}}.
\]
Let $\varepsilon>0$ be a given small real number. 
By the definition of $B^\mu(f)$, we have 
\[
\sum_{j=1}^{\deg g} \frac{\cW^{2\mu}}{(Jf\circ f_j^{-1})\cdot \cW^{2\mu}\circ Df_j^*}\le B^{2\mu}(f)+\varepsilon
\]
on the outside of a sufficiently large neighborhood of the zero section depending on $\varepsilon>0$. For a negative number $\kappa<\min\{0,\mu-1/2\}$, we can take a large constant $C>0$ so that 
\begin{equation}\label{B}
 B^{2\mu}(f)+\varepsilon + C\cdot \cW^{2\kappa}-\sum_{j=1}^{\deg g} \frac{\cW^{2\mu}}{(Jf\circ f_j^{-1})\cdot \cW^{2\mu}\circ Df_j^*} 
 \end{equation}
 is non-negative. 
 Hence, applying G\r{a}rding's inequality \cite[\S 2.8]{TaylorText81} to the pseudo-differential operator with symbol \eqref{B},  we obtain
\[
\mathrm{Re} \langle b_i, A b_i\rangle \le (B^{2\mu}(f)+\varepsilon) \|b_i\|_{L^2}^{2}+ C\|v_i\|^2_{H^{\mu-(1/2)}}.
\]
From the relation $v_j:=u_j\circ f_j^{-1}$ and  the coordinate change formula \cite[\S 2.5]{TaylorText81}, we see that 
\begin{align*}
\| b_j \|_{L^2}
&\le \left\|\OP(\cW^\mu \circ Df_j^*)\left(\frac{v_{j}}{\sqrt{Jf\circ f_j^{-1}}}\right)\right\|_{L^2}+C \|v_{j}\|_{H^{\mu-1}}\\
&\le \left\|\frac{(\OP(\cW^\mu ) u_j) \circ f_j^{-1}}{\sqrt{Jf\circ f_j^{-1}}}\right\|_{L^2}+C'\|u_{j}\|_{H^{\mu-1}}\\
&=\|\OP(\cW^\mu)u_{j}\|_{L^2}+C' \|u_{j}\|_{H^{\mu-1}}
\end{align*}
for some constants $C,C'>0$.
Finally note that we have 
\[
\|v_j\|_{H^{\mu'}}\le C\|u_j\|_{H^{\mu'}}
\] 
for any $\mu'\in \real$ and also that 
\[
\|u_j\|_{H^{\mu'}}\le C\|u_j\|_{H^{\mu''}}
\]
when $\mu'\le \mu''$ where $C$ may depend on $\mu'$ and $\mu''$. 
Therefore, summarizing the argument above, we have that, for any $\max\{0,\mu-(1/2)\}\le\mu'<\mu$ and $\varepsilon>0$, there exists $C>1$ such that
\[
\|Pu\|_{H^\mu}\le \sqrt{(B^{2\mu}(f)+\varepsilon)}\|u\|_{H^\mu}+C\|u\|_{H^{\mu'}}. 
\]

By using a partition of unity and the product formula (\cite[\S  2.4]{TaylorText81}), we see that the same inequality holds for arbitrary $u\in C^\infty(M)$. 
Since $H^\mu(M)$ for $\mu>0$ is compactly embedded in $H^{\mu'}(M)$ and since $\varepsilon>0$ is arbitrary,  the last inequality and Henion's theorem \cite{Hennion93} tell that the essential spectral radius of $P$ on $H^\mu(M)$ is bounded by $\sqrt{B^{2\mu}(f)}$. Finally noting that the argument above remains valid when $f$ is replaced by its iterates $f^n$, we obtain the conclusion of the theorem. 

\section{A conclusive remark}
It is possible to generalize our result without much effort. In the last section, we consider the weight functions $\mathcal{W}^\mu$ for $\mu>0$. But we could consider a more general class of positive function  in the place of $\mathcal{W}^\mu$. For instance, we may set 
\[
\mathcal{W}^{\boldsymbol{\mu}}(x,\xi)=\langle |\xi|\rangle^{\boldsymbol{\mu}([(x,\xi)])}
\]
for a $C^\infty$ positive-valued function $\boldsymbol{\mu}:\mathbb{P}T^*M\to \real$ on the projectivized cotangent bundle $\mathbb{P}T^*M$ and prove Theorem\footnote{In the case where $\boldsymbol{\mu}$ takes negative values, the argument in the proof of Corollary is not valid.} with setting 
\[
B^{\boldsymbol{\mu}}(f)= \sup_{(x,\xi)\in T^*M}
\sum_{(y,\eta)\in T^*M: f^{\dag}(y,\eta)=(x,\xi)}
\frac{\cW^{\boldsymbol{\mu}}(x,\xi)}{|Jf(y)|\cdot \cW^{\boldsymbol{\mu}}(y,\eta)}.
\]
We refer \cite{FRS,1706.09307} for the treatment of the pseudo-differential operators $\OP(\mathcal{W}^{\boldsymbol{\mu}})$. 
It will be an interesting subject to find a reasonable but non-trivial class of dynamical systems for which $\lim_{n\to \infty}B^{2\boldsymbol{\mu}}(f^n)<1$. 

\appendix
\section{Virtually expanding maps that are not expanding}
Let us recall the map $f:\mathbb{T}^2\to \mathbb{T}^2$ in Example, defined by \eqref{example}. Let $\mu>0$ be arbitrary. Below we  prove  $B^{\mu}(f)<1$ assuming that $m$ is sufficiently large. 

Take an arbitrary point $(q,\xi)\in T^{*}\mathbb{T}^2$ with $\xi\neq 0$ and write $\xi=(\xi_x,\xi_y)$. We will show $b^{\mu}(q,\xi)<c<1$ for some constant $c$ independent of $(q,\xi)$.  If $\xi_y=0$, we see $
b^\mu(q,\xi)=m^{-\mu}<1$. Hence we may assume $\xi_y \neq 0$. Further, since $b^\mu(q,\xi)=b^{\mu}(q,\beta \xi)$ for any $\beta\neq 0$, we may and do assume $\xi_y=1$ and write $\xi=(\xi_x,1)$. 

Consider $(p,\eta)\in (f^\dag)^{-1}(q,\xi)$ and write $p=(x,y)$ and $\eta=(\eta_x,\eta_y)$. Then we have $\eta_y=1$ and 
\begin{align*}
\begin{bmatrix} \xi_x\\ 1\end{bmatrix}=(D_{p}f^{-1})^{*}\begin{bmatrix} \eta_x\\ 1\end{bmatrix}
&=\begin{bmatrix}1/m& -2\pi  \cdot \sin 2\pi  x\\ 0 & 1
\end{bmatrix}\begin{bmatrix} \eta_x\\ 1\end{bmatrix}=\begin{bmatrix}\eta_x/m  -2\pi  \cdot \sin 2\pi  x\\  1
\end{bmatrix}
\end{align*}
This implies that, if 
\begin{equation}\label{eq:cone}
|\xi_x + 2\pi \sin 2\pi x|\ge 5\pi /m,
\end{equation}
 we have $|\eta_x|\ge 5\pi$ and 
 \[
 \frac{|\xi|}{|\eta|}\le \frac{\sqrt{1+(\eta_x/m+2\pi)^2}}{\sqrt{1+\eta_x^2}}\le 
 \frac{\sqrt{1+(2\pi)^2}}{\sqrt{1+(4\pi)^2}}
 \]
provided that $m$ is sufficiently large. (Notice  that the last inequality holds obviously if $|\eta_x|$ is sufficiently large.) 

From the observation above, the sum on the right hand side of \eqref{eq:b} restricted to those elements $(p,\eta)\in (f^{\dag})^{-1}(q,\xi)$ satisfying the condition \eqref{eq:cone} is bounded by $\sqrt{1+(2\pi)^2}/\sqrt{1+(4\pi)^2}$. 
The cardinality of the remaining elements of $(f^{\dag})^{-1}(q,\xi)$ is bounded by $C\sqrt{m}$ with $C$ a constant independent of $m$. Hence the corresponding partial sum is bounded by 
\[
C\sqrt{m} \cdot \frac{\max_{w\in M} \|(D_wf^*)^{-1}\|^{\mu}}{m}\le \frac{C_\mu}{\sqrt{m}}
\]
where $C_\mu>1$ is a constant depending on $\mu$. 
Therefore, letting $m$ be sufficiently large, we obtain 
\[
b^{\mu}(q,\xi)< 
\left|\frac{1+(2\pi)^2}{1+(4\pi)^2}\right|^{\mu/2}+ \frac{C_\mu}{\sqrt{m}}<\left|\frac{1+(2\pi)^2}{1+(3\pi)^2}\right|^{\mu/2}<1
\]
uniformly for $(q,\xi)$.
Thus we have obtained $B^{\mu}(f)<1$. 


\bibliography{mybib}
\bibliographystyle{amsplain} 


\end{document}